\documentclass{amsart}
\usepackage{amsmath,amsfonts,amsthm,amssymb,verbatim,enumerate,quotes,graphicx,mathtools}
\usepackage{color}
\newcommand{\comments}[1]{}
\newtheorem{theorem}{Theorem}
\newtheorem*{theoremA}{Theorem A}
\newtheorem{definition}{Definition}
\newtheorem{lemma}{Lemma}
\newtheorem{corollary}{Corollary}
\newtheorem{example}{Example}
\newtheorem*{remark}{Remark}
\newtheorem*{remarks}{Remarks}

\newtheorem{proposition}{Proposition}
\def \isnatural {\in\mathbb{N}}

\def\R{\mathbb{R}}
\def\C{\mathbb{C}}

\def \spw {spider's web}

\newcommand{\tef}{transcendental entire function}

\newcommand\qfor{\quad\text{for }}

\def\firstbound{L}
\def\secondbound{L'}
\makeatletter
\def\blfootnote{\xdef\@thefnmark{}\@footnotetext}
\makeatother
\begin{document}
%
%
%
%
\title[Periodic domains of quasiregular maps]{Periodic domains of quasiregular maps}
\author{Daniel A. Nicks, \, David J. Sixsmith}
\address{School of Mathematical Sciences \\ University of Nottingham \\ Nottingham
NG7 2RD \\ UK}
\email{Dan.Nicks@nottingham.ac.uk}
\address{School of Mathematical Sciences \\ University of Nottingham \\ Nottingham
NG7 2RD \\ UK}
\email{David.Sixsmith@nottingham.ac.uk}
%
%
%
%
\begin{abstract}
We consider the iteration of quasiregular maps of transcendental type from $\R^d$ to $\R^d$. We give a bound on the rate at which the iterates of such a map can escape to infinity in a periodic component of the quasi-Fatou set. We give examples which show that this result is best possible. Under an additional hypothesis, which is satisfied by all uniformly quasiregular maps, this bound can be improved to be the same as those in a Baker domain of a transcendental entire function. 

We construct a quasiregular map of transcendental type from $\R^3$ to $\R^3$ with a periodic domain in which all iterates tend locally uniformly to infinity. This is the first example of such behaviour in a dimension greater than two. 

Our construction uses a general result regarding the extension of biLipschitz maps. In addition, we show that there is a quasiregular map of transcendental type  from $\R^3$ to $\R^3$ which is equal to the identity map in a half-space. 
\end{abstract}
\maketitle
%
%
%
%
\blfootnote{Both authors were supported by Engineering and Physical Sciences Research Council grant EP/L019841/1.}
\section{Introduction}
The \emph{Fatou set} $F(f)$ of an entire function $f$ is the set of points $z\in\C$ such that $\{f^k\}_{k\isnatural}$ is a normal family in some neighbourhood of $z$. The \emph{Julia set} $J(f)$ is the complement in $\mathbb{C}$ of $F(f)$, and the \emph{escaping set} is defined by 
\begin{equation}
\label{Ifdef}
I(f) := \{ z \in \C : f^k(z) \rightarrow\infty \text{ as } k\rightarrow\infty \}.
\end{equation}
We refer to \cite{MR1216719} for further information on the properties of these sets.

If $U$ is a component of $F(f)$, and there exists a least integer $p$ such that $f^p(U)~\subset~U$, then $U$ is called \emph{$p$--periodic}. If, in addition, $U \cap I(f) \ne \emptyset$, then $U$ is called a \emph{Baker domain}, and, by normality, the iterates of $f$ tend locally uniformly to infinity in $U$. Baker domains 
have been a subject of much study; see, for example, \cite{MR980793, MR1830901, MR1931601, MR1688496, MR2247639} and the survey article \cite{MR2458809}.

The first example of a {\tef} with a Baker domain is 
\begin{equation}
\label{Fatoueq}
h(z) := z + e^{-z} + 1,
\end{equation}
given by Fatou \cite[Exemple I]{MR1555220}. It can be seen that $h$ has a Baker domain containing a right half-plane. The rate at which points in the Baker domain tend to infinity is slow; $h$ eventually behaves ``like'' $z \mapsto z + 1$. This observation may be quantified for a Baker domain of any {\tef} as follows \cite[Theorem~1]{MR2247639}. The first proof of this result, in the case $p=1$, was given by Baker \cite{MR980793}.
\begin{theoremA}
Suppose that $f$ is a {\tef} and that $U$ is a $p$-periodic Baker domain of $f$. Then, for $x \in U$,
\begin{equation}
\label{TAeq1}
\log|f^{kp}(x)| = O(k) \text{ as } k \rightarrow\infty.
\end{equation}
Also, if $X \subset U$ is compact, then there exist $C > 1$ and $k_0 \isnatural$ such that
\begin{equation}
\label{TAeq2}
|f^{kp}(x')| \leq C|f^{kp}(x)|, \qfor x, x' \in X, k \geq k_0.
\end{equation}
Finally, if $\xi_0 \in U$ and $\Gamma_0 \subset U$ is a curve joining $\xi_0$ to $f^{p}(\xi_0)$ such that $0~\notin~\bigcup_{k=0}^\infty f^{kp}(\Gamma_0)$, then there is a constant $C>1$ such that 
\begin{equation}
\label{TAeq3}
\frac{1}{C} |x| \leq |f^p(x)| \leq C|x|, \qfor x \in \bigcup_{k=0}^\infty f^{kp}(\Gamma_0).
\end{equation}
\end{theoremA}

In this paper we use some techniques from our earlier paper \cite{SixsmithNicks1} to extend this study to more than two (real) dimensions. Suppose that $d \geq 2$, and that $f : \R^d \to \R^d$ is a \emph{quasiregular map}; we define a quasiregular map in Section~\ref{Sdefs}. Here we are mainly interested in the case that $f$ is of \emph{transcendental type}; in other words, $f$ has an essential singularity at infinity.

In the context of quasiregular maps, a definition of the Julia set different to that used in complex dynamics is required. We follow \cite{MR3009101, MR3265283}, and define the \emph{Julia set $J(f)$} to be the set of points $x \in \R^d$ such that
\begin{equation}
\label{Jdef}
\operatorname{cap} \left(\R^d\setminus\bigcup_{k=1}^\infty f^k(U)\right) = 0, \qfor \text{every neighbourhood } U \text{ of } x.
\end{equation}
Recall that if $S \subset \R^d$, then cap $S = 0$ means that $S$ is, in a sense that can be made precise, a ``small'' set; we refer to \cite{MR1238941, MR950174} for a full definition. If $f$ is a quasiregular map of transcendental type, then $J(f)$ is infinite \cite[Theorem 1.1]{MR3265283}. 

Following \cite{SixsmithNicks1}, we define the \emph{quasi-Fatou set $QF(f)$} as the complement in $\R^d$ of the Julia set; note that it follows from \cite[Theorem 1.2]{MR3265283} that if $f$ is a {\tef}, then $QF(f) = F(f)$. If $U$ is a component of $QF(f)$, and there exists a least integer $p$ such that $f^p(U) \subset U$, then $U$ is called \emph{a $p$--periodic domain}. If $U$ is not periodic, but there is an integer $k$ such that $f^k(U)$ is contained in a periodic domain, then $U$ is called \emph{pre-periodic}. In the remaining case, $U$ is called \emph{wandering}. The escaping set $I(f)$ is defined by the obvious modification to (\ref{Ifdef}).

Since it follows from \cite[Theorem 3.1]{MR759304} that a Baker domain of a {\tef} is full, our first main result is a generalisation of (\ref{TAeq1}) to quasiregular maps. Here $K_I(f)$ is the \emph{inner dilatation} of $f$; we refer to Section~\ref{Sdefs} for a definition. We say that a domain is \emph{full} if it has no bounded complementary components; otherwise we say that it is \emph{hollow}. 
Also, $\log^{+}$ is the function defined by $$\log^+(x) := \log(\max\{1,x\}), \qfor x \in \R.$$
\begin{theorem}
\label{Tspeed}
Suppose that $f : \R^d \to \R^d$ is a quasiregular map of transcendental type, and that $U$ is a full $p$-periodic domain. Then
\begin{equation}
\label{Tspeedeq1}
\limsup_{k\rightarrow\infty} \frac{1}{k}\log^+\log|f^{kp}(x)| \leq  \log K_I(f^p), \qfor x \in U.
\end{equation}
\end{theorem}

In the other direction, in Section~\ref{Sexample} we give examples which show that the conclusion of Theorem~\ref{Tspeed} is best possible; in particular, the bound (\ref{Tspeedeq1}) cannot be improved to (\ref{TAeq1}) in the quasiregular case.

\begin{remarks}\normalfont
\begin{enumerate}
\item We cannot assume that if a component, $U$, of the quasi-Fatou set meets the escaping set, then the iterates tend to infinity locally uniformly in $U$. In Example~\ref{Example3} below we give a quasiregular map of $\R^3$ with a periodic domain that meets, but is not contained in, the escaping set. \\  
\item If $f$ is a quasiregular map that is not of transcendental type, and the degree of $f$ is greater than $K_I(f)$, then the Julia set of $f$ is non-empty \cite[Theorem 1.1]{MR3009101}, and it follows from, for example, \cite[Theorem 1.4]{ETS:9408364}, that there exists a unique unbounded component of $QF(f)$, which is hollow and periodic. The rate at which the iterates tend to infinity in this component is described in \cite{ETS:9408364}.
\end{enumerate}
\end{remarks}

A quasiregular map $f : \R^d \to \R^d$ is called \emph{uniformly $K$-quasiregular} if all the iterates of $f$ are $K$-quasiregular; this definition was introduced in \cite{MR1404085}. It is possible to give a result stronger than Theorem~\ref{Tspeed} for maps with this property. However, for $d\geq 3$, there are currently no known examples of uniformly quasiregular maps of transcendental type.  

We introduce the following definition. Suppose that $f : \R^d \to \R^d$ is a quasiregular map, and that $U \subset \R^d$ is a domain. We say that $f$ is \emph{locally uniformly quasiregular} in $U$ if, for each $x \in U$, there is a neighbourhood $V$ of $x$ such that all the iterates of $f$ restricted to $V$ are $K$-quasiregular for some $K$, which may depend on $x$. Clearly, if $f: \R^d \to \R^d$ is uniformly quasiregular, then $f$ is locally uniformly quasiregular in $\R^d$. We note also that the function $f : \R^3 \to \R^3$ constructed in the proof of Theorem~\ref{Texists} below is locally uniformly quasiregular in the unique component $U=QF(f)$.

We have the following stronger result for maps with this property.
\begin{theorem}
\label{Tspeed2}
Suppose that $f : \R^d \to \R^d$ is a quasiregular map of transcendental type, that $U$ is a full $p$-periodic domain that meets $I(f)$, and that $f$ is locally uniformly quasiregular in $U$. Then the iterates of $f$ tend locally uniformly to infinity in $U$, and the conclusions of Theorem~A all hold.
\end{theorem}
\begin{remark}\normalfont
If $f$ is a {\tef}, then $f$ is uniformly quasiregular, and so the proof of Theorem~\ref{Tspeed2} gives a new technique to prove Theorem~A.
\end{remark}

It was shown in \cite{MR1684251} that a Baker domain of a {\tef} cannot meet the \emph{fast escaping set} $A(f)$; we defer a definition of this set to Section~\ref{Sdefs}. Since, as mentioned earlier, a Baker domain of a {\tef} is full, one direction of the following result can be seen as a generalisation of this fact to quasiregular maps of transcendental type. The other direction, when combined with Lemma~\ref{lemm:inA} below, shows that the hypothesis in the statement of Theorem~\ref{Tspeed} that $U$ is full seems to be required. We observe, however, that there is no known example of a quasiregular map of transcendental type with a hollow domain that is periodic or pre-periodic.
\begin{theorem}
\label{T2}
Suppose that $f : \mathbb{R}^d \to \mathbb{R}^d$ is a quasiregular map of transcendental type, and that $U$ is a component of $QF(f)$ that is periodic or pre-periodic. Then $U$ meets $A(f)$ if and only if $U$ is hollow.
\end{theorem} 
In Section~\ref{Sspeed} we state a corollary to Theorem~\ref{T2} for functions that do not grow too slowly. 

Fatou showed that the Baker domain of the function $h$ defined in (\ref{Fatoueq}) is equal to the whole Fatou set $F(h)$. Our second main result is that there is a quasiregular map of $\R^3$ with analogous properties.
\begin{theorem}
\label{Texists}
There exists a quasiregular map of transcendental type $f : \mathbb{R}^3~\to~\mathbb{R}^3$ such that $QF(f)$ consists of a single full domain, $U$, in which all iterates of $f$ tend locally uniformly to infinity. Moreover, $f$ is locally uniformly quasiregular in $U$. 
\end{theorem}

In the proof of Theorem~\ref{Texists} we define a function with the following property, which may be of independent interest; in fact the functions constructed in the proofs of Theorem~\ref{Texists} and Theorem~\ref{Texistshalf} differ only by a translation.
\begin{theorem}
\label{Texistshalf}
There exists a quasiregular map of transcendental type $g : \mathbb{R}^3~\to~\mathbb{R}^3$ that is equal to the identity map in a half-space.
\end{theorem}

The structure of this paper is as follows. First, in Section~\ref{Sdefs} we recall the definitions of quasiregularity, a useful metric, and the fast escaping set, and we give some known results required in the rest of the paper. In Section~\ref{Sspeed} we prove Theorem~\ref{Tspeed}, Theorem~\ref{Tspeed2}, Theorem~\ref{T2} and Corollary~\ref{Cor1}. 
In Section~\ref{Sconst} we prove Theorem~\ref{theorem:construction}, which concerns the extension of biLipschitz maps, and is used later in the proofs of Theorem~\ref{Texists} and Theorem~\ref{Texistshalf}. In Section~\ref{SZor} we define a Zorich map, in Section~\ref{Sexistshalf} we prove Theorem~\ref{Texistshalf}, and then in Section~\ref{Sexists} we prove Theorem~\ref{Texists}. Finally, in Section~\ref{Sexample} we give the examples mentioned above. 

\subsection*{Notation}
In much of this paper we work in $\R^3$. If $x \in \R^3$, then we adopt the notation $x = (x_1, x_2, x_3)$ without comment; for example, $\{ x_3 > 0\}$ denotes the half-space $\{ x = (x_1, x_2, x_3) \in \R^3 : x_3 > 0 \}$. 

If $x, y$ are distinct points of $\R^d$, then we write $L(x,y)$ for the straight line containing $x$ and $y$ extended to infinity in both directions, and $l(x,y)$ for the straight line segment from $x$ to $y$.

We denote the Euclidean distance from a point $x$ to a set $U \subset \R^d$ by $$\operatorname{dist}(x,U) := \inf_{y\in U} |x - y|.$$

We denote by $B(a,r)$ the open ball of Euclidean radius~$r$, centred at a point $a \in\R^d$.

%
%
%
%
%
%
%
\section{Definitions and background results}
\label{Sdefs}
We refer to \cite{MR1238941, MR950174} for a detailed treatment of quasiregular maps, and recall here the definition, and some properties used in this paper.

Suppose that $d\geq 2$, that $G \subset \R^d$ is a domain, and that $1 \leq p < \infty$. The \emph{Sobolev space} $W^1_{p,loc}(G)$ consists of those functions $f : G \to \R^d$ for which all first order weak partial derivatives exist and are locally in $L^p$. We say that $f$ is \emph{quasiregular} if $f \in W^1_{d,loc}(G)$ is continuous, and there exists $K_O \geq 1$ such that
\begin{equation}
\label{KOeq}
|D f(x)|^d \leq K_O J_f(x) \quad a.e.
\end{equation}
Here $D f(x)$ denotes the derivative,
$$
|D f(x)| := \sup_{|h|=1} |Df(x)(h)|
$$
is the norm of the derivative, and $J_f (x)$ denotes the Jacobian determinant. We also define
$$
\ell(Df(x)) := \inf_{|h|=1} |Df(x)(h)|.
$$

If $f$ is quasiregular, then there also exists $K_I \geq 1$ such that 
\begin{equation}
\label{KIeq}
K_I \ell(D f(x))^d \geq J_f(x) \quad a.e.
\end{equation}
The smallest constants $K_O$ and $K_I$ for which (\ref{KOeq}) and (\ref{KIeq}) hold are called the \emph{outer and inner dilatation} of $f$ and denoted by $K_O (f)$ and $K_I (f)$. 
We say that $f$ is \emph{$K$-quasiregular} if $\max\{K_I (f),K_O (f)\} \leq K$, for some $K \geq 1$.

A homeomorphism that satisfies (\ref{KOeq}) and (\ref{KIeq}) with $|J_f(x)|$ in place of $J_f(x)$ is called \emph{quasiconformal}. Note that, by this definition, a quasiconformal map may be orientation-reversing.

If $f$ and $g$ are quasiregular maps, and $f$ is defined in the range of $g$, then $f \circ g$ is quasiregular and \cite[Theorem II.6.8]{MR1238941}
\begin{equation}
\label{Keq}
K_I(f \circ g) \leq K_I(f)K_I(g).
\end{equation}

Many properties of holomorphic functions extend to quasiregular maps; in particular, non-constant quasiregular maps are open and discrete.

We use a result on the growth of the maximum modulus of a quasiregular map of transcendental type; see \cite[Lemma 3.4]{MR2248829}, \cite[Corollary 4.3]{MR1799670}. Here $M(R,f)$ denotes the \emph{maximum modulus function} $$M(R,f) := \max_{|x| = R} |f(x)|, \qfor R>0.$$ 
\begin{lemma}
\label{Blemm}
Suppose that $f : \mathbb{R}^d \to \mathbb{R}^d$ is a quasiregular map of transcendental type. Then
$$
\lim_{r\rightarrow\infty} \frac{\log M(r, f)}{\log r} = \infty.
$$
\end{lemma}
As in \cite{SixsmithNicks1}, we use a certain conformal invariant that is useful when working with quasiregular maps. The definition of this invariant is complicated, so we simply state the properties we use, and refer to \cite{MR950174} for full details.

Let $G \subset \R^d$ be a domain. Vuorinen \cite[p.103]{MR950174} defines a function
$$
  \mu_G : G \times G \to \R
$$
with the properties that $\mu_G$ is a conformal invariant, and is a metric if cap $\partial G > 0$. It is noted that if $D \subset G$ is a domain, then
\begin{equation}
\label{smallerdomain}
\mu_D(x,y) \geq \mu_G(x,y), \qfor  x, y \in D.
\end{equation}

We use the following \cite[Theorem 10.18]{MR950174}. 
\begin{lemma}
\label{l3}
Suppose that $f : G \to \mathbb{R}^d$ is a non-constant quasiregular map. Then
$$
\mu_{f(G)}(f(a), f(b)) \leq K_I(f) \mu_G(a, b), \qfor a, b \in G.
$$
\end{lemma}
We also use the following estimate for $\mu_G$ \cite[Theorem 8.31]{MR950174}. Here, the boundary of $G$ taken in $\overline{\R^d}$ is denoted by $\partial_\infty G$.
\begin{lemma}
\label{muestimate}
Suppose that $G$ is a proper subdomain of $\mathbb{R}^d$ and that $\partial_\infty G$ is connected. Then there exists a constant $c_d$, which depends only on $d$, such that
$$
\mu_G(a, b) \geq c_d \log\left(1 + \frac{|a-b|}{\min\{\operatorname{dist}(a, \partial G), \operatorname{dist}(b, \partial G)\}}\right), \qfor a, b \in G.
$$
\end{lemma}
We use the fact, which is straightforward to prove, that if $G$ is a full subdomain of $\R^d$, then $\partial_\infty G$ is connected. \\

Suppose that $f : \R^d \to \R^d$ is quasiregular. 
The \emph{fast escaping set} is a subset of the escaping set, and is defined by
\begin{equation}
\label{Adef}
A(f) := \{x : \text{there exists } \ell \isnatural \text{ such that } |f^{k+\ell}(x)| \geq M^k(R, f), \text{ for } k \isnatural\}.
\end{equation}
Here $M^k(R, f)$ denotes the $k$th iterate of $M(R, f)$ with respect to the first variable, and $R > 0$ can be taken to be any value such that $M^k(R,f)\rightarrow\infty$ as $k\rightarrow\infty$. For a quasiregular map $f$ of transcendental type, this definition was first used in \cite{MR3215194}, where it was shown that $A(f)$ is independent of the choice of $R$.

We need the following properties of the fast escaping set, which combine \cite[Theorem 1.2]{MR3215194} and \cite[Theorem~1.1 and Theorem 1.2]{MR3265357}. 
\begin{lemma}
\label{lemm:JA}
Suppose that $f : \R^d \to \R^d$ is a quasiregular map of transcendental type. Then $A(f)$ is non-empty and has no bounded components, and $J(f) \subset \partial A(f)$. If, in addition, 
\begin{equation}
\label{12A}
\liminf_{R\rightarrow\infty} \frac{\log\log M(R,f)}{\log\log R} = \infty,
\end{equation}
then $J(f) = \partial A(f)$.
\end{lemma} 
%
%
%
%
%
%
\section{Results concerning rate of escape}
\label{Sspeed}
\begin{proof}[Proof of Theorem~\ref{Tspeed}]
Suppose that $f : \R^d \to \R^d$ is a quasiregular map of transcendental type, and that $U$ is a full $p$-periodic domain. Choose $x_0 \in \partial U$. \\ 

We note first the following. Suppose that $V \subset U$ is a domain. Let $c_d$ be the constant in Lemma~\ref{muestimate}. It follows, by (\ref{smallerdomain}) and Lemma~\ref{muestimate}, that
\begin{align*}
\mu_V(a, b) &\geq \mu_{U}(a, b) \\
            &\geq c_d \log\left(1+\frac{|a-b|}{|b-x_0|}\right) \\ 
            &\geq c_d \log\frac{|a-x_0|}{|b-x_0|}, \qfor a,b \in V.
\end{align*}
Hence 
\begin{equation}
\label{musep}
|a-x_0| \leq |b-x_0| \exp\left(\frac{\mu_{V}(a, b)}{c_d}\right), \qfor a, b \in V.
\end{equation}

Clearly, $f^{kp}(U) \subset f^{(k-1)p}(U) \subset U$, for $k \isnatural$. Suppose that $x \in U$. We deduce by (\ref{Keq}), (\ref{musep}) and Lemma~\ref{l3} that
\begin{align*}
|f^{kp}(x)-x_0| &\leq |f^{(k-1)p}(x)-x_0| \exp\left(\frac{\mu_{f^{(k-1)p}(U)}(f^{kp}(x), f^{(k-1)p}(x))}{c_d}\right) \\
                 &\leq |f^{(k-1)p}(x)-x_0| \exp\left(\frac{ K_I(f^{(k-1)p}) \mu_U(f^p(x),x)}{c_d}\right) \\
                 &\leq |f^{(k-1)p}(x)-x_0| \exp\left(\frac{ K_I(f^p)^{k-1} \mu_U(f^p(x),x)}{c_d}\right), \qfor k\isnatural.
\end{align*}
Hence
\begin{equation}
\label{almostdone}
|f^{kp}(x)-x_0| \leq |x-x_0|  \exp \left(\frac{\mu_U(f^p(x),x)}{c_d}\sum_{j=0}^{k-1} K_I(f^p)^{j}\right), \qfor k\isnatural,
\end{equation}
and (\ref{Tspeedeq1}) follows.
\end{proof}
\begin{proof}[Proof of Theorem~\ref{Tspeed2}]
Suppose that $f : \R^d \to \R^d$ is a quasiregular map of transcendental type, that $U$ is a full $p$-periodic domain that meets $I(f)$, and that $f$ is locally uniformly quasiregular in $U$. Choose $x_0 \in \partial U$. \\

Suppose that $x \in U$. Let $D$ be a bounded domain, compactly contained in $U$, such that $x \in D$ and $f^p(x) \in D$. It follows, by a compactness argument, that $f$ is uniformly quasiregular in $D$. This, in turn, implies that $f$ is uniformly quasiregular in the domain
\begin{equation}
\label{Veq}
V := \bigcup_{k \geq 0} f^{kp}(D) \subset U.
\end{equation}

We deduce by (\ref{musep}) and Lemma~\ref{l3} that there exists $K \geq 1$ such that
\begin{align*}
|f^{kp}(x)-x_0| 
                &\leq |f^{(k-1)p}(x)-x_0| \exp\left(\frac{ K \mu_V(f^p(x),x)}{c_d}\right), \qfor k\isnatural,
\end{align*}
in which case
\begin{equation*}
|f^{kp}(x)-x_0| \leq |x-x_0|  \exp \left(\frac{kK\mu_V(f^p(x),x)}{c_d} \right), \qfor k\isnatural.
\end{equation*}

Equation (\ref{TAeq1}) follows. \\ 

To prove (\ref{TAeq2}), suppose that $X \subset U$ is compact. Let $D$ be a bounded domain, compactly contained in $U$, such that $X \cup f^p(X) \subset D$, and let $V$ be as in (\ref{Veq}). Set
$$
L := \sup_{x, x' \in X} \mu_V (x, x').
$$
Note that it follows from the compactness of $X$, together with \cite[Theorem 1]{MR752460} (and see also \cite[Theorem 2]{MR0302904}), that $L < \infty$. We deduce by (\ref{musep}) and Lemma~\ref{l3} that there exists $K \geq 1$ such that
\begin{align*}
|f^{kp}(x')-x_0| &\leq |f^{kp}(x)-x_0| \exp\left(\frac{ K L}{c_d}\right), \qfor x, x' \in X, \ k\isnatural.
\end{align*}
Equation (\ref{TAeq2}) follows. The fact that the iterates of $f$ tend locally uniformly to infinity in $U$ also follows. \\

Finally, let $\xi_0 \in U$ and let $\Gamma_0 \subset U$ be a curve joining $\xi_0$ to $f^{p}(\xi_0)$ such that $0~\notin~\bigcup_{k=0}^\infty f^{kp}(\Gamma_0)$. If $x \in \bigcup_{k=0}^\infty f^{kp}(\Gamma_0)$, then we can write $x = f^{kp}(\xi)$, for some $\xi \in \Gamma_0$ and $k \geq 0$. We then obtain (\ref{TAeq3}) by setting $X = \Gamma_0 \cup f^p(\Gamma_0)$ and applying (\ref{TAeq2}) to the points $\xi$ and $f^p(\xi)$.
\end{proof}
%

To prove Theorem~\ref{T2} we require the following, which is a generalisation of \cite[Theorem 2.2 (c)]{Rippon01102012} to quasiregular maps of transcendental type.
\begin{lemma}
\label{lemm:inA}
Suppose that $f : \R^d \to \R^d$ is a quasiregular map of transcendental type, and that $x \in A(f)$. Then
\begin{equation}
\label{xinAeq}
\lim_{k\rightarrow\infty} \frac{1}{k}\log \log |f^{k}(x)| = \infty.
\end{equation}
\end{lemma} 
\begin{proof}
Let $R>0$ be sufficiently large that $M^k(R, f)\rightarrow\infty$ as $k\rightarrow\infty$. It follows by Lemma~\ref{Blemm} that
\begin{equation}
\label{Mgrows}
\lim_{k\rightarrow\infty} \frac{1}{k} \log \log M^k(R, f) = \infty.
\end{equation}

Suppose that $x \in A(f)$. Then, by (\ref{Adef}), there exists $\ell\isnatural$ such that, for all sufficiently large $k\isnatural$,
\begin{equation*}
\frac{1}{k}\log \log |f^{k}(x)| \geq \frac{1}{k}\log \log M^{k-\ell}(R, f).
\end{equation*}
The result follows by (\ref{Mgrows}).
\end{proof} 
\begin{remark}\normalfont
When $f$ is a {\tef}, points that satisfy (\ref{xinAeq}) are said to \emph{zip to infinity}. The set of points with this property was studied in \cite{MR1767705}.
\end{remark} 
\begin{proof}[Proof of Theorem~\ref{T2}]
Suppose that $f : \mathbb{R}^d \to \mathbb{R}^d$ is a quasiregular map of transcendental type, and that $U$ is a component of $QF(f)$ that is periodic or pre-periodic. \\

Suppose first that $U \cap A(f) \ne \emptyset$. By way of contradiction, we suppose that $U$ is full. There exists $k\isnatural$ such that the component of $QF(f)$ that contains $f^k(U)$, $V$ say, is $p$-periodic. By \cite[Corollary 5.2]{SixsmithNicks1} $V$ is full, and it is easy to see that $V$ meets $A(f)$. Take $x \in V \cap A(f)$, in which case, by Lemma~\ref{lemm:inA}, equation (\ref{xinAeq}) holds. This is in contradiction to (\ref{Tspeedeq1}). \\

In the other direction, suppose that $U$ is hollow. We deduce by \cite[Theorem 1.3]{SixsmithNicks1} that $U$ is unbounded. It then follows by \cite[Theorem 1.4]{SixsmithNicks1} that $U$ has no unbounded complementary components. Since, by Lemma~\ref{lemm:JA}, $A(f)$ has at least one unbounded component, it follows that $U$ meets $A(f)$.
\end{proof}
As mentioned in the introduction, Theorem~\ref{T2} has the following corollary. 
\begin{corollary}
\label{Cor1}
Suppose that $f : \mathbb{R}^d \to \mathbb{R}^d$ is a quasiregular map of transcendental type such that (\ref{12A}) is satisfied. Suppose also that $U$ is a hollow component of $QF(f)$. Then $U \subset A(f)$.
\end{corollary}
\begin{proof}[Proof of Corollary~\ref{Cor1}]
%
If $U$ is bounded, then it follows by \cite[Theorem 1.3]{SixsmithNicks1} that $U \subset A(f)$. Hence we can assume that $U$ is unbounded, in which case, by \cite[Theorem 1.4]{SixsmithNicks1}, $U$ is periodic. It follows by Theorem~\ref{T2} that $U$ meets $A(f)$. It then follows by the final part of Lemma~\ref{lemm:JA} that $U \subset A(f)$, as required.
\end{proof}
%
%
%
%
%
%
%
\section{Radial extension of biLipschitz maps}
\label{Sconst}
The goal of this section is to prove a result, Theorem~\ref{theorem:construction} below, that is used to construct the functions described in Theorem~\ref{Texists} and Theorem~\ref{Texistshalf}. Although there are related results in the literature -- see, for example, \cite{MR3365742, MR3195912, MR595177} -- the authors are not aware of a reference for this result. 

We require some preliminary results and notation. For generality, we give these results in $\R^d$.
\begin{definition}
A domain $A\subset \R^d$ is called a \emph{star domain} if there exists $a\in A$ such that, for all $x \in A\setminus\{a\}$, the line segment $l(a,x)$ lies in $A$. Such a point is called a \emph{star centre} of $A$.
\end{definition}

We need to be able to choose a star centre with a certain property, and so we introduce the following definition.
\begin{definition}
\label{def:1}
Suppose that $A\subset \R^d$ is a \emph{bounded} star domain with star centre $a$. 
We say that $a$ is a \emph{non-tangential} star centre if there exist $\theta\in (0, \pi/4)$ and $\epsilon > 0$ with the following property. For all  points $w, v \in \partial A$ such that $0 < |w-v| < \epsilon$, the acute angle between the lines $L(a,w)$ and $L(w,v)$ is greater than $\theta$.
\end{definition}
Note that an equivalent definition is that for each point $w\in\partial A$, the boundary of $A$ only meets the double-cone with vertex $w$, axis through $a$, and aperture $2\theta$, either at the point $w$ or outside the ball $B(w, \epsilon)$. 

It follows from the definitions that if $A$ is a bounded star domain, $a$ is a non-tangential star centre of $A$, and $x \in A\setminus\{a\}$, then $L(a,x)$ meets $\partial A$ at exactly two points. One of these points, which we denote by $\hat{x}$, is such that $x \in l(a, \hat{x})$. We define a surjection $\psi_{A} : \overline{A}\setminus\{a\} \to \partial A$ by
$$
\psi_A(x) := 
\begin{cases}
\hat{x}, \qfor x \in A\setminus\{a\}, \\ x, \qfor x \in \partial A.
\end{cases}
$$

Clearly $\psi_A$ depends on the choice of star centre, but we suppress this for clarity; the choice of star centre will always be clear from the context. We require the following property of $\psi_A$.
\begin{lemma}
\label{glemma}
Suppose that $A \subset \R^d$ is a bounded star domain with a non-tangential star centre $a$. Then $\psi_{A}$ is locally Lipschitz. Indeed, there exist $\eta \in (0, 1)$ and $T > 0$ such that for all $\xi \in \overline{A}\setminus\{a\}$,
\begin{equation}
\label{glemmalip}
|\psi_A(x) - \psi_A(y)| \leq \frac{T}{|\xi-a|}|x - y|, \qfor x, y \in B(\xi, \eta |\xi - a|) \cap \overline{A}.
\end{equation}
\end{lemma}
\begin{proof}
\begin{figure}
	\includegraphics[width=14cm,height=10cm]{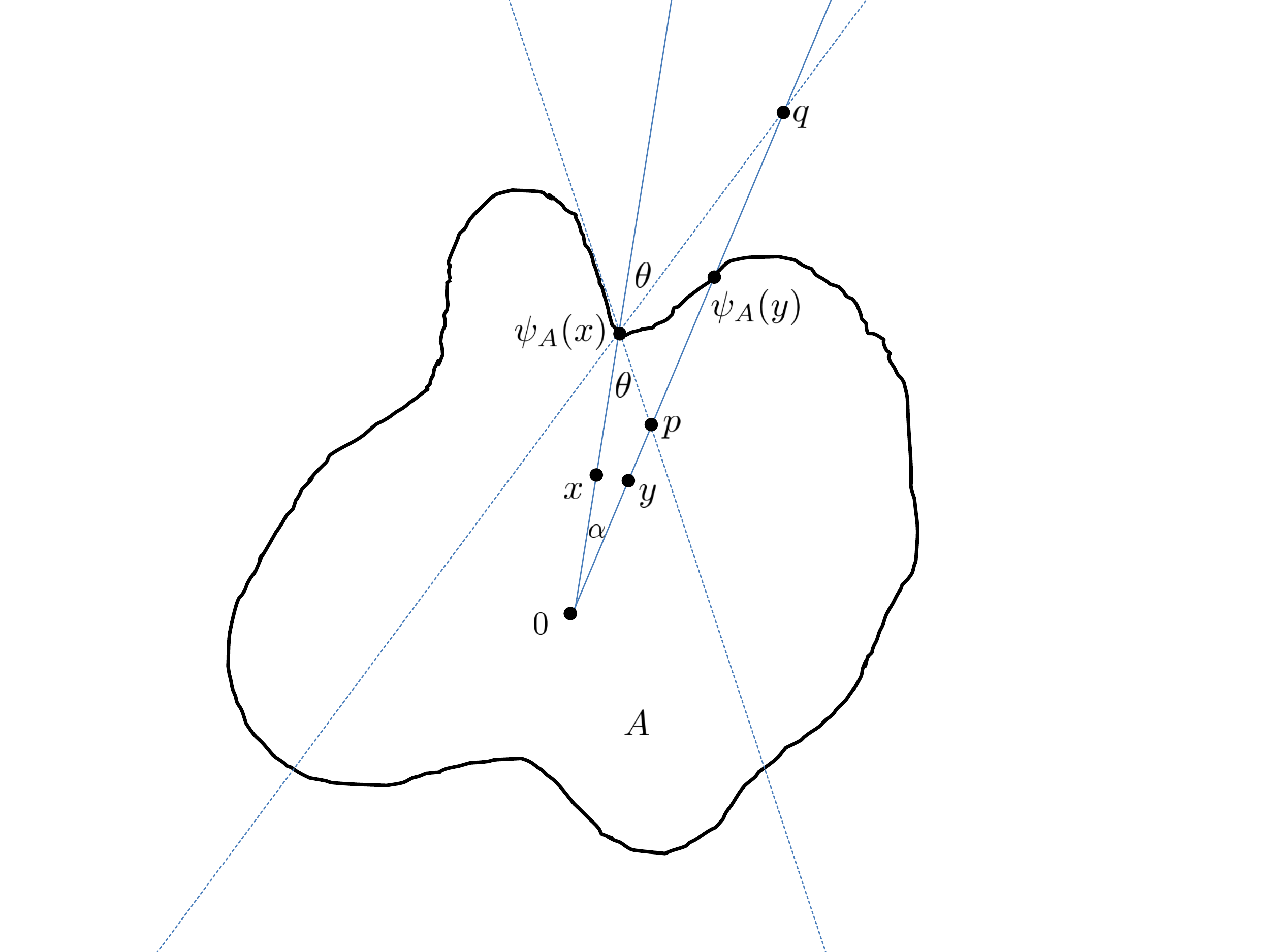}
  \caption{The construction from Lemma~\ref{glemma}}\label{fig1}
\end{figure}
For simplicity of notation, we can assume that $a = 0$. Let $\theta$ and $\epsilon$ be the constants from Definition~\ref{def:1}. Let $\eta > 0$ be sufficiently small that
$$
\eta < \min \left\{\frac{1}{2}, \frac{1}{4}\sin \left(\theta/2\right), \frac{\epsilon \sin\left(\theta/2\right)}{4\max_{w \in\partial A} |w|} \right\}.
$$

Fix $\xi\in\overline{A}\setminus\{0\}$. Suppose that $x, y \in B(\xi, \eta|\xi|) \cap \overline{A}$. We can assume that $\psi_A(x) \ne \psi_A(y)$, as otherwise there is nothing to prove. Let $p$ and $q$ be the points where $L(0,y)$ meets the double cone with vertex $\psi_A(x)$, axis through $0$, and aperture $2\theta$, and let $\alpha$ be the acute angle between $L(0,x)$ and $L(0,y)$; see Figure~\ref{fig1}. 

We note that the first two constraints on $\eta$ imply first that 
\begin{equation}
\label{eq:sinalpha}
\sin \alpha \leq \frac{|x - y|}{|x|}\leq \frac{2\eta|\xi|}{(1-\eta)|\xi|} \leq 4 \eta, 
\end{equation}
and then that $2\alpha < \theta$. 

It follows, by a geometric construction and by (\ref{eq:sinalpha}), that the constraints on $\eta$ ensure that $p$ and $q$ both lie within the ball $B(\psi_A(x), \epsilon)$. Hence, by Definition~\ref{def:1}, $L(0,y)$ meets $\partial A$ in $B(\psi_A(x), \epsilon)$, and therefore $\psi_A(y) \in \ell(p,q)$.

It follows, by a similar geometric calculation, that
\begin{align*}
|\psi_A(x) - \psi_A(y)| &\leq \max\{|\psi_A(x) - p|, |\psi_A(x) - q|\} \\
                        &= \frac{|\psi_A(x)|\sin \alpha}{\min\{\sin(\theta-\alpha), \sin(\pi-\theta-\alpha)\}}\\
                        &\leq |x - y|\frac{2\max_{w \in \partial A} |w|}{|\xi|\sin (\theta/2)}.
\end{align*}
The result follows, with $T = 2\csc (\theta/2)\max_{w \in \partial A} |w|$.
\end{proof}
We now give the main result of this section.
\begin{theorem}
\label{theorem:construction}
Suppose that $A, B \subset \R^d$ are bounded star domains with non-tangential star centres $a, b$ respectively. Suppose also that $f : \partial A \to \partial B$ is a biLipschitz surjection. Radially extend $f$ to a function from $\overline{A}$ to $\overline{B}$, which we continue to denote by $f$, by defining
\begin{equation}
\label{fdef}
f(x) := 
\begin{cases}
b + \frac{|x - a|}{|\psi_A(x) - a|}(f(\psi_A(x)) - b), \qfor x \in A\setminus\{a\} \\
b, \quad\quad\quad\quad\quad\quad\quad\quad\quad\quad\quad\text{ }\qfor x = a.
\end{cases}
\end{equation}
Then $f$ is a biLipschitz surjection.
\end{theorem}
\begin{proof}
Note that the definition of $f$ is symmetric, in the following sense. Since $f$ is biLipschitz on $\partial A$, it has a biLipschitz inverse $\phi$ on $\partial B = f(\partial A)$. It follows from (\ref{fdef}) that $\phi$ radially extends to an inverse of $f$ on $\overline{B}$ with
\begin{equation}
\label{phidef}
\phi(x) := 
\begin{cases}
a + \frac{|x - b|}{|\psi_B(x) - b|}(\phi(\psi_B(x)) - a), \text{ }\qfor x \in B\setminus\{b\} \\
a, \quad\quad\quad\quad\quad\quad\quad\quad\quad\quad\quad\quad\qfor x = b.
\end{cases}
\end{equation}

We prove that $f$ is Lipschitz on $\overline{A}$. Applying the same argument to $\phi$ shows that $\phi$ is Lipschitz on $\overline{B}$. The fact that $f$ is biLipschitz then follows. \\

For simplicity of notation, we can assume that $a=b=0$. Let 
$$
M := \sup_{w \in \overline{A} \cup \overline{B}} |w| < \infty \quad\text{and}\quad \delta := \inf_{w \in \partial A \cup \partial B} |w| > 0,
$$
and let $L$ be the Lipschitz constant of $f$ on $\partial A$.

Since 
$$
|f(x) - f(0)| = \frac{|x|}{|\psi_A(x)|}|f(\psi_A(x))| \leq \frac{M}{\delta}|x|,
$$
we see that $f$ is Lipschitz at the origin.

Now take $x, y \in \overline{A}$. We can assume, without loss of generality, that $|x| \geq |y| > 0$. Let $w := \frac{f(\psi_A(x))}{|f(\psi_A(x))|}$ and note that $|w| = 1$. Then
\begin{align}
|f(x) - f(y)| &= \left|\frac{|x|}{|\psi_A(x)|}f(\psi_A(x)) - \frac{|y|}{|\psi_A(y)|}f(\psi_A(y))\right| \nonumber \\
              &= \frac{|f(\psi_A(x))|}{|\psi_A(x)|} \left||x|w - \frac{|y||\psi_A(x)|}{|\psi_A(y)|}\frac{f(\psi_A(y))}{|f(\psi_A(x))|}\right|.
\label{daneq}
\end{align}

We consider two cases. First, suppose that $|x-y|\geq\eta|x|$, where $\eta \in (0,1)$ is the constant from Lemma~\ref{glemma}. Using (\ref{daneq}) then gives that
\begin{align*}
\frac{|f(x) - f(y)|}{|x-y|} 
             &\leq \frac{|f(x) - f(y)|}{\eta|x|} \\
             &\leq \frac{|f(\psi_A(x))|}{\eta|\psi_A(x)|} \left(1 + \frac{|y||\psi_A(x)||f(\psi_A(y))|}{|x||\psi_A(y)||f(\psi_A(x))|}\right) \\
             &\leq \frac{M}{\eta\delta}\left(1 + \frac{M^2}{\delta^2}\right).
\end{align*}

On the other hand, suppose instead that $|x-y| < \eta|x|$, in other words that $y~\in~B(x, \eta|x|)$. By (\ref{daneq})
\begin{align*}
&|f(x) - f(y)| \\
&= \frac{|f(\psi_A(x))|}{|\psi_A(x)|}\left|(|x| - |y|)w + \frac{|y|(|\psi_A(y)| - |\psi_A(x)|)}{|\psi_A(y)|}w + \frac{|y||\psi_A(x)|}{|\psi_A(y)||f(\psi_A(x))|}(f(\psi_A(x)) - f(\psi_A(y)))\right|\\
&\leq \frac{M}{\delta}\left(|x - y| + \frac{|y|}{\delta}|\psi_A(y) - \psi_A(x)| + \frac{|y|M}{\delta^2}|f(\psi_A(x)) - f(\psi_A(y))|\right)\\
&\leq \frac{M}{\delta}\left(|x - y| + \frac{|y|}{\delta}\left(1 + \frac{ML}{\delta}\right)|\psi_A(y) - \psi_A(x)|\right).
\end{align*}

Now an application of Lemma~\ref{glemma}, with $\xi = x$, yields that
\begin{align*}
|f(x) - f(y)| &\leq \frac{M}{\delta}\left(|x - y| + \left(1 + \frac{ML}{\delta}\right)\frac{|y|T}{|x|\delta}|x-y|\right) \\
              &\leq \frac{M}{\delta}\left(1 + \left(1 + \frac{ML}{\delta}\right)\frac{T}{\delta}\right) |x - y|,
\end{align*}
completing the proof that $f$ is Lipschitz on $\overline{A}$.
\end{proof}
%
%
%
\section{A modified Zorich map, $Z$, and the function $Id + Z$}
\label{SZor}
Zorich \cite{MR0223569} introduced a family of quasiregular maps of transcendental type from $\R^3$ to $\R^3\setminus\{0\}$, now known as Zorich maps, which can be seen as analogues of the exponential function; see, for example, \cite[Section 6.5.4]{MR1859913} for a detailed description of these maps. First we give the definition of a modified Zorich map; this map is slightly easier to work with in our setting than any of the standard Zorich maps.

The standard definition of a Zorich map in $\R^3$ starts with a map from a square to a hemisphere. Instead, we use a map from a square to the upper faces of a square based pyramid. To this end, set $$M(x_1, x_2) := \max\{|x_1|, |x_2|\},$$ and define a biLipschitz map $h$ from the square $\left[-1, 1\right]^2$ to the upper faces of the square based pyramid with vertex $(0,0,1)$ and base $[-1,1]^2 \times \{0\}$ by setting
$$
h(x_1, x_2) := \left(x_1, x_2, 1 - M(x_1, x_2)\right).
$$

The rest of the definition of a Zorich map proceeds in the usual way. In an infinite square cylinder we define
$$
Z : \left[-1,1\right]^2 \times \mathbb{R} \to \{x_3 \geq 0\}
$$
by setting
$$
Z(x_1,x_2,x_3) := e^{x_3} h(x_1,x_2).
$$

This is then extended to a map $Z : \mathbb{R}^3 \to \mathbb{R}^3 \setminus \{0\}$ by reflections in the usual way; in particular, when the domain (a  square cylinder) is reflected in a face, the image is reflected in the plane $\{ x_3 = 0\}$. It can be shown that $Z$ is quasiregular, and periodic in the $x_1$ and $x_2$ directions with period $4$.

We observe that although $Z$ is such that $Z([1,3]^2\times \R) = Z([-1,1]^2\times \R)$, the function $Z$ operates slightly differently on these two domains. This is because the two reflections used to define $Z$ in $[1,3]^2\times \R$ induce a rotation. \\ 

Next, we define a function $F : \R^3 \to \R^3$ by 
\begin{equation}
\label{Feq}
F(x) := x + Z(x).
\end{equation}

We consider the properties of $F$ in a \emph{fundamental half-beam}. This is defined as a domain of the form
\begin{equation}
\label{Bnmeq}
\mathcal{B}_{n,m} = \{ |x_1 - 2n| < 1, |x_2 - 2m| < 1, x_3 > \firstbound \}, \qfor n, m \in \mathbb{Z},
\end{equation}
where $L>0$ is a constant that is fixed in the following proposition. 
\begin{proposition}
\label{prop:Fbeam}
There exist $\firstbound > 1$ and $K>1$ such that if $\mathcal{B}$ is a fundamental half-beam, then the following both hold:
\begin{enumerate}[(a)]
\item $F$ is $K$-quasiregular in $\mathcal{B}$;
\item $|F(x) - F(y)| \geq 32|x - y|$, for $x, y \in \mathcal{B}$.
\end{enumerate}
\end{proposition}
\begin{proof}
First, we consider the fundamental half-beam $\mathcal{B}_{0,0}$. 
Then, 
\begin{equation*}
DF(x) = 
\begin{cases}
\left(
\begin{array}{ccc}%
	e^{x_3} + 1 & 0 & x_1 e^{x_3} \\
	0 & e^{x_3} + 1 & x_2 e^{x_3} \\
	-\frac{x_1}{|x_1|}e^{x_3} & 0 & (1-|x_1|)e^{x_3} + 1
\end{array} 
\right), \qfor 0 < |x_2| < |x_1| < 1, \\
\left(
\begin{array}{ccc}%
	e^{x_3} + 1 & 0 & x_1 e^{x_3} \\
	0 & e^{x_3} + 1 & x_2 e^{x_3} \\
	0 & -\frac{x_2}{|x_2|}e^{x_3} & (1-|x_2|)e^{x_3} + 1
\end{array} 
\right), \qfor 0 < |x_1| < |x_2| < 1.
\end{cases}
\end{equation*}
Hence there exists $\firstbound>1$ such that
\begin{equation}
\label{Jeq}
J_F(x) \geq \frac{1}{2} e^{3 x_3} \quad\text{and}\quad |DF(x)| \leq 7 e^{x_3}, \quad \text{a.e. for } x \in \mathcal{B}_{0,0}.
\end{equation}

By a similar calculation, it can be seen that the inequality (\ref{Jeq}) holds in all other fundamental half-beams. The first part of the proposition follows. \\
 
Increasing $\firstbound$ if necessary, we can deduce by a similar, but slightly simpler, calculation of the derivative of $Z$ that 
$$
\ell(DZ(x)) > 33, \quad\text{a.e. for  } x \in \{ x_3 > L \}.
$$
Suppose that $\mathcal{B}$ is a fundamental half-beam. We note that $Z$ is injective in $\mathcal{B}$; this is an immediate consequence of the definition of $Z$. As in the proof of \cite[(2.3)]{MR2797689}, it follows by considering the inverse function of $Z|_{\mathcal{B}}$, that
\begin{equation*}
|F(x) - F(y)| \geq |Z(x) - Z(y)| - |x - y| \geq 32|x - y|, \qfor x, y \in \mathcal{B},
\end{equation*}
as required.
\end{proof}
It is straightforward to show that $F$ is not quasiregular in $\R^3$; there are domains in which $F$ is orientation-reversing. However, we have the following immediate corollary of Proposition~\ref{prop:Fbeam}.
\begin{corollary}
\label{corr:Fgood}
The function $F$ is $K$-quasiregular in $\{ x_3 > \firstbound \}$.
\end{corollary}

Finally in this section, we record some elementary properties of the function $F$. 
%
First, 
\begin{equation}
\label{periodiceq1}
F(x+ c) = F(x)+ c, \qfor x \in \R^3,\ c \in \{(4,0,0), (0,4,0)\}.
\end{equation}

Second, let $R_1$ denote reflection in the plane $\{ x_1 = 2\}$, and let $R_2$ denote reflection in the plane $\{ x_2 = 2\}$. Then $F$ commutes with these functions; in other words
\begin{equation}
\label{roteq}
F \circ R_i = R_i \circ F, \qfor i \in \{1,2\}.
\end{equation}

It is immediate that the identity map $Id$ also enjoys these properties.
%
%
%
%
%
%
\section{Proof of Theorem~\ref{Texistshalf}}
\label{Sexistshalf}
In this section we prove Theorem~\ref{Texistshalf} by constructing a quasiregular map on $\R^3$ which is equal to $F = Id + Z$ in a half-space, equal to $Id$ in another half-space, and given by an interpolation between these half-spaces. We start the interpolation by defining a quasiregular map on a certain cuboid.

Let $\firstbound>1$ be the constant from Proposition~\ref{prop:Fbeam}, and let $A$ be the cuboid
$$
A := \{ 0 \leq x_1 \leq 2, \ 0 \leq x_2 \leq 2, \ 0 \leq x_3 \leq \firstbound \}.
$$

We will define a quasiregular map $g : A \to \R^3$ with the following properties. If $x = (x_1, x_2, x_3) \in A$ and $g(x) = (g_1(x), g_2(x), g_3(x))$, then:
\begin{enumerate}[(A)]
\item $g(x) = Id(x), \qfor x_3 = 0$;\label{A}
\item $g(x) = F(x), \qfor x_3 = \firstbound$;\label{B}
\item $g_1(x) = 0, \qfor x_1 = 0$;\label{C}
\item $g_1(x) = 2, \qfor x_1 = 2$;\label{D}
\item $g_2(x) = 0, \qfor x_2 = 0$;\label{E}
\item $g_2(x) = 2, \qfor x_2 = 2$.\label{F}
\end{enumerate}

Before giving the details of the construction of $g$ on $A$, we show how to extend $g$ to a quasiregular map on $\R^3$, of transcendental type, and equal to the identity map in a half-space.

So, for the moment, assume that we can define a quasiregular map $g : A \to \R^3$ with properties (\ref{A})--(\ref{F}). We first extend the definition of $g$ in $A$ to three similar cuboids. First we set
$$
g(x) = (R_1 \circ g \circ R_1)(x), \qfor x \in \{ 2 \leq x_1 \leq 4, \ 0 \leq x_2 \leq 2, \ 0 \leq x_3 \leq \firstbound \}.
$$
This extension is continuous because of property (\ref{D}). Second we set
$$
g(x) = (R_2 \circ g \circ R_2)(x), \qfor x \in \{ 0 \leq x_1 \leq 2, \ 2 \leq x_2 \leq 4, \ 0 \leq x_3 \leq \firstbound \}.
$$
This extension is continuous because of property (\ref{F}). Finally, we set
$$
g(x) = (R_2 \circ R_1 \circ g \circ R_1 \circ R_2)(x), \qfor x \in \{ 2 \leq x_1 \leq 4, \ 2 \leq x_2 \leq 4, \ 0 \leq x_3 \leq \firstbound \}.
$$
This extension is continuous because of properties (\ref{D}) and (\ref{F}). This defines a quasiregular map in 
$$
B := \{ 0 \leq x_1 \leq 4, \ 0 \leq x_2 \leq 4, \ 0 \leq x_3 \leq \firstbound \}.
$$

Moreover, we have the following. Suppose that $x = (x_1, x_2, x_3) \in B$. Then: 
\begin{enumerate}[(i)]
\item $g(x) = Id(x)$, for $x_3 = 0$\label{A1}. This holds by property (\ref{A}) above, and by construction.
\item $g(x) = F(x)$, for $x_3 = \firstbound$\label{B1}. This holds by property (\ref{B}) above, and then by construction and by (\ref{roteq}).
\item $g(4, x_2, x_3) = g(0, x_2, x_3) + (4, 0, 0)$. \label{D1} This holds by construction and by property (\ref{C}) above.
\item $g(x_1, 4, x_3) = g(x_1, 0, x_3) + (0, 4, 0)$. \label{F1} This holds by construction and by property (\ref{E}) above.
\end{enumerate}

We next extend the definition of $g$ to $\{ 0 \leq x_3 \leq \firstbound\}$. For $a \in \R$, we write $\tilde{a} := a \operatorname{mod } 4$; in other words, $\tilde{a} \in [0,4)$ is such that $a - \tilde{a}$ is a multiple of $4$. We then set
$$
g(x_1, x_2, x_3) = g(\tilde{x}_1, \tilde{x}_2, x_3) + (x_1 - \tilde{x}_1, x_2 - \tilde{x}_2, 0), \qfor 0 \leq x_3 \leq L.
$$

This construction gives a quasiregular map on $\{ 0 \leq x_3 \leq \firstbound\}$ because of properties (\ref{D1}) and (\ref{F1}) above. Moreover, it follows from (\ref{periodiceq1}) that properties (\ref{A1}) and (\ref{B1}) above are still satisfied.

Finally we extend $g$ to the whole of $\R^3$ by setting $g(x_1, x_2, x_3) = Id(x_1, x_2, x_3)$, for $x_3 < 0$ and $g(x_1, x_2, x_3) = F(x_1, x_2, x_3)$, for $x_3 > \firstbound$. It follows from properties (\ref{A1}) and (\ref{B1}) above that $g$ is a continuous map of $\R^3$.

Clearly $g$ is equal to the identity map in $\{ x_3 < 0 \}$, and quasiregular there. Moreover $g$ is quasiregular in $\{ 0 \leq x_3 \leq \firstbound\}$, by construction. Finally, $g$ is quasiregular in $\{ x_3 > \firstbound\}$ by Corollary~\ref{corr:Fgood}. It remains, therefore, to construct the function $g$ in $A$. \\

In fact, we achieve this 
in two stages. We partition $A$ into two cuboids
$$
A' := \{ 0 \leq x_1 \leq 2, \ 0 \leq x_2 \leq 2, \ 0 \leq x_3 \leq 1 \},
$$
and
$$
A'' := \{ 0 \leq x_1 \leq 2, \ 0 \leq x_2 \leq 2, \ 1 \leq x_3 \leq \firstbound \},
$$
and we define $g$ separately in each.

We ensure that the definition of $g$ in $A'$ satisfies property (\ref{A}) and properties (\ref{C})--(\ref{F}). The definition of $g$ in $A'$ includes a ``folding'' which is needed to accommodate the property of $Z$ in the domain $[1,3]^2 \times \R$, mentioned in Section~\ref{SZor}.

We ensure that the definition of $g$ in $A''$ satisfies property (\ref{B}), properties (\ref{C})--(\ref{F}), and agrees with the previous definition in $A' \cap A''$. 

%
%
%
\subsection{The construction in $A'$}
\label{Adashdef}
In this subsection we show how to define the map $g$ in $A'$. First we define a set of points of $A'$, and fix their images under $g$. These are specified in Table~\ref{table:1}. We use the convention that, for example, $P_1$ (resp. $P_\firstbound$) is a translation of the point $P_0$ by one (resp. $\firstbound$) units in the direction of the third coordinate.

We note that not all points in Table~\ref{table:1} are vertices of $A'$. However, for simplicity, we call points such as $T_1$ and $X_1$ \emph{vertices}, the $20$ line segments, such as $l(T_1, X_1)$, shown in Figure~\ref{fig:cuboid} \emph{edges}, and the four small squares such as that with vertices $T_1, X_1, W_1, P_1$ \emph{faces}.
 
\begin{table}[ht]
\centering
 \begin{tabular}{||c c c c c c||} 
 \hline
 Point & Coordinate & Image under $g$ & Point & Coordinate & Image under $g$ \\ [0.5ex] 
 \hline\hline
 $P_0$ & $(0,0,0)$ & $(0,0,0)$ & $P_1$ & $(0,0,1)$ & $(0,0,4)$ \\ [0.25ex] 
 \hline
 $Q_0$ & $(0,2,0)$ & $(0,2,0)$ & $Q_1$ & $(0,2,1)$ & $(0,2,3.5)$\\ [0.25ex] 
 \hline
 $R_0$ & $(2,2,0)$ & $(2,2,0)$ & $R_1$ & $(2,2,1)$ & $(2,2,0.5)$ \\ [0.25ex] 
 \hline
 $S_0$ & $(2,0,0)$ & $(2,0,0)$ & $S_1$ & $(2,0,1)$ & $(2,0,-0.4)$ \\ [0.25ex] 
 \hline
 & & & $T_1$ & $(0,1,1)$ & $(0,4,4)$ \\ [0.25ex] 
 \hline
 & & & $U_1$ & $(1,2,1)$ & $(4.5,2,2)$ \\ [0.25ex] 
 \hline
 & & & $V_1$ & $(2,1,1)$ & $(2,4,-0.4)$ \\ [0.25ex] 
 \hline
 & & & $W_1$ & $(1,0,1)$ & $(6,0,2)$ \\ [0.25ex] 
 \hline
 & & & $X_1$ & $(1,1,1)$ & $(6,4,2)$ \\ [0.25ex] 
 \hline
 \end{tabular}
\caption{Points of $A'$, and their images}
\label{table:1}
\end{table}

\begin{figure}[ht]
	\centering
	\includegraphics[width=14cm,height=10cm]{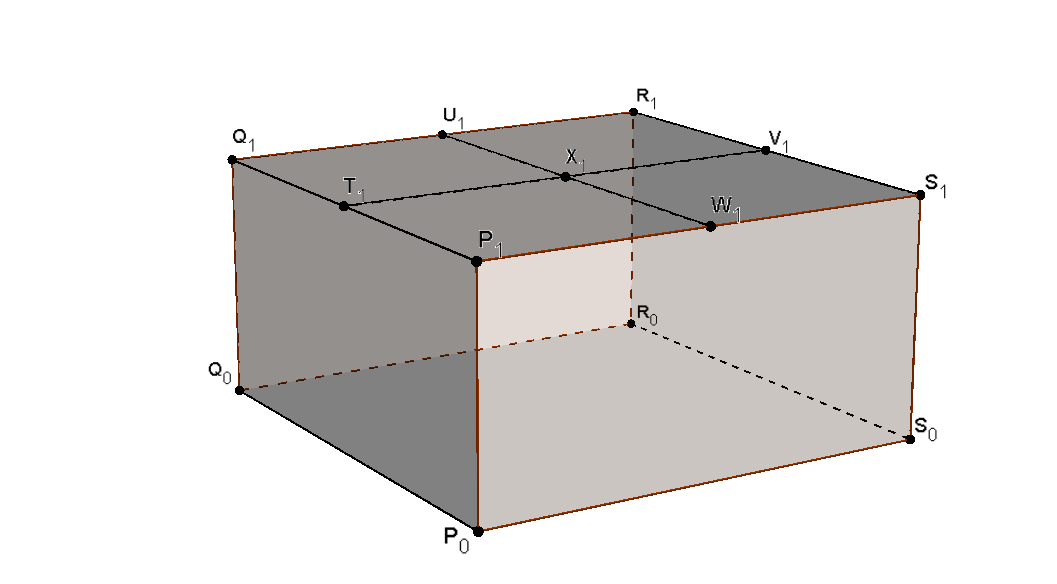}
	\caption{The cuboid $A'$ and points from Table~\ref{table:1}. Note that these points divide the top face of $A'$ into four squares.}
  \label{fig:cuboid}
\end{figure}

The images of the points in Table~\ref{table:1} have been chosen in such a way that they define a polyhedron, which we denote by $B'$, which has nine faces:
\begin{itemize}
\item A square in the plane $\{ x_3 = 0 \}$, defined by the vertices $P_0, Q_0, R_0, S_0$, which is also a face of $A'$;
\item Four quadrilaterals:
\begin{itemize}
\item One in the plane $\{ x_1 + 3x_3 = 12 \}$, defined by the images of the vertices $P_1, T_1, X_1, W_1$;
\item One in the plane $\{ 4x_1 - 3x_2 + 12x_3 = 36 \}$, defined by the images of the vertices $T_1, Q_1, U_1, X_1$;
\item One in the plane $\{ -12x_1 + 9x_2 + 20x_3 = 4 \}$, defined by the images of the vertices $U_1, R_1, V_1, X_1$;
\item One in the plane $\{ 3x_1 - 5x_3 = 8 \}$, defined by the images of the vertices $V_1, S_1, W_1, X_1$;
\end{itemize}
\item Four pentagons:
\begin{itemize}
\item One in the plane $\{ x_1 = 0\}$, defined by the images of the vertices $P_0, P_1, T_1, Q_1, Q_0$;
\item One in the plane $\{ x_2 = 0\}$, defined by the images of the vertices $P_0, P_1, W_1, S_1, S_0$;
\item One in the plane $\{ x_1 = 2\}$, defined by the images of the vertices $R_0, R_1, V_1, S_1, S_0$;
\item One in the plane $\{ x_2 = 2\}$, defined by the images of the vertices $Q_0, Q_1, U_1, R_1, R_0$.
\end{itemize}
\end{itemize}

\begin{figure}[ht]
	\centering
	\includegraphics[width=14cm,height=10cm]{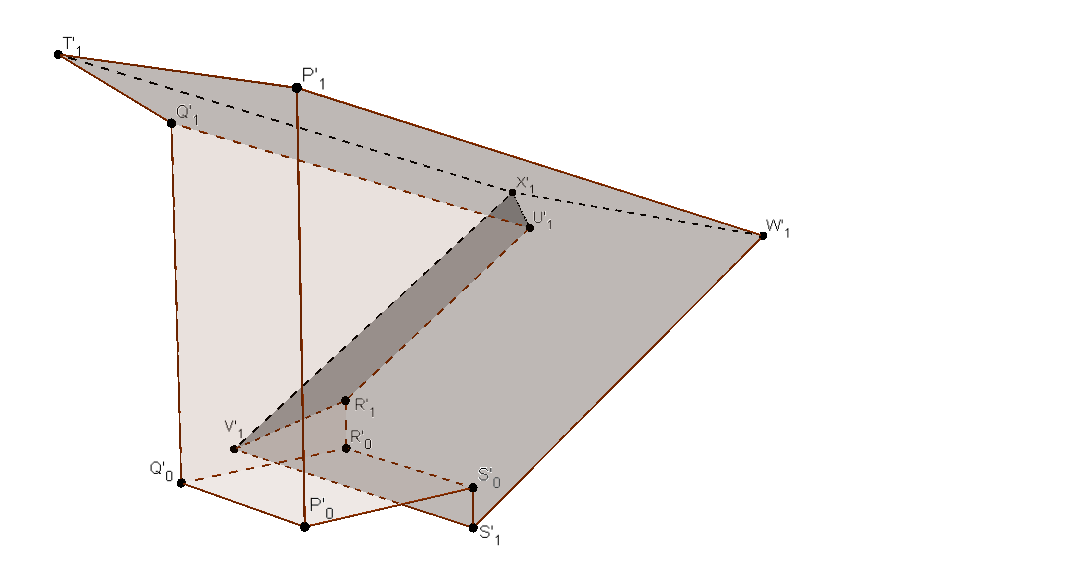}
	\caption{A view of the polyhedron $B' = g(A')$, using points from Table~\ref{table:1}. Note that, for example, $P'_0$ denotes $g(P_0)$. We have shaded the images of faces which correspond to the four squares in the upper face of $A'$.}
	\label{fig:inter}
\end{figure}

We aim to construct the function $g$ so that it maps the cuboid $A'$ onto the polyhedron $B'$; see Figure~\ref{fig:inter}. Using the images of the vertices specified in Table~\ref{table:1}, we first define $g$ from each edge of $A'$ to the corresponding edge of $B'$ simply as an affine map.

Since the four points $P_0, Q_0, R_0, S_0$ map to themselves, we can set $g = Id$ on the square defined by these four points. Hence property~(\ref{A}) is satisfied.

We map the remaining eight faces of $A'$ to the corresponding eight faces of $B'$ in the following way. Let, for example, $\alpha$ be the interior of the face of $A'$ with vertices $T_1, X_1, W_1, P_1$ and let $\beta$ be the interior of the face of $B'$ with vertices the images of these points. It is easy to see that $\alpha$ and $\beta$ are star domains with non-tangential star centres. By considering $\alpha$ and $\beta$ as each lying in a copy of $\R^2$, we apply Theorem~\ref{theorem:construction} to extend the definition of $g$ to a biLipschitz surjection from $\overline{\alpha}$ to $\overline{\beta}$. We apply this process to each pair of corresponding faces of $A'$ and $B'$. The resulting function, $g$, from the boundary of $A'$ to the boundary of $B'$ is a biLipschitz surjection.

Now, it can been seen from Figure~\ref{fig:inter} -- and confirmed by an elementary calculation -- that the interior of $B'$ is a star domain, with a non-tangential star centre at, for example, $(5, 1, 2)$. Also, the interior of $A'$ is clearly a star domain with a non-tangential star centre at, for example, $(1,1,0.5)$. Hence we apply Theorem~\ref{theorem:construction} once again to extend the definition of $g$ to a biLipschitz surjection from $A'$ to $B'$. 

Since $g$ is biLipschitz, it follows (see, for example, \cite[p. XI]{MR950174}) that $g$ is quasiconformal, and is either orientation-preserving or orientation-reversing. It can be seen that $g$ is orientation-preserving. Hence $g$ is quasiregular.

The points $P_0, P_1, T_1, Q_1, Q_0$ and the images of these points all lie in the plane $\{ x_1 = 0 \}$. Hence property (\ref{C}) is satisfied. Similarly the points $R_0, R_1, V_1, S_1, S_0$ and the images of these points all lie in the plane $\{ x_1 = 2 \}$. Hence property (\ref{D}) is satisfied. It is easy to see that properties (\ref{E}) and (\ref{F}) are satisfied, for similar reasons.
%
%
%
\subsection{The construction in $A''$}
In this subsection we define the map $g$ in $A''$. The construction is complicated. First, it is necessary to subdivide $A''$ into four smaller cuboids;
\begin{itemize}
\item $A''_1 := \{ 0 \leq x_1 \leq 1,   \ 0 \leq x_2 \leq 1, \ 1 \leq x_3 \leq \firstbound \}$,
\item $A''_2 := \{ 1 \leq x_1 \leq 2, \ 0 \leq x_2 \leq 1, \ 1 \leq x_3 \leq \firstbound \}$,
\item $A''_3 := \{ 0 \leq x_1 \leq 1,   \ 1 \leq x_2 \leq 2, \ 1 \leq x_3 \leq \firstbound \}$,
\item $A''_4 := \{ 1 \leq x_1 \leq 2, \ 1 \leq x_2 \leq 2, \ 1 \leq x_3 \leq \firstbound \}$.
\end{itemize}

For each $i \in \{1,2,3,4\}$, we define a biLipschitz map $g$ from $A''_i$ to a compact set  $B''_i$ which is the closure of a star domain. To achieve this, we adopt the approach used in the previous section, by defining a biLipschitz map from $\partial A''_i$ to $\partial B''_i$, and then extending the map to $\overline{A''_i}$ using Theorem~\ref{theorem:construction}.

Each cuboid $A''_i$ has a face which lies in the plane $\{ x_3 = 1 \}$ and a face which lies in the plane $\{ x_3 = \firstbound \}$; we call these two faces the \emph{end faces}. Each cuboid $A''_i$ has a face which lies in the plane $\{x_1 = 1\}$ and a face which lies in the plane $\{x_2 = 1\}$; we call these faces the \emph{interior faces}. The remaining two faces we call the \emph{exterior faces}. We give the definition of $g$ on these three types of face as follows.
\subsubsection{The end faces}
\label{endfaces}
Suppose that $i\in\{1,2,3,4\}$. We have already fixed the definition of $g$ on $\partial A''_i  \cap \{ x_3 = 1\}$ in Subsection~\ref{Adashdef}. In order to satisfy property (\ref{B}), we set $$g(y) = F(y), \qfor y \in \partial A''_i \cap \{ x_3 = \firstbound\}.$$ 

Note that the image of each of the squares $\partial A''_i \cap \{ x_3 = \firstbound \}$ consists of two triangles. For example, $\partial A''_1 \cap \{ x_3 = \firstbound \}$ maps to the triangles with vertices $g(P_\firstbound), g(W_\firstbound), g(X_\firstbound)$ and $g(P_\firstbound), g(X_\firstbound), g(T_\firstbound)$.
%
\subsubsection{The exterior faces}
\label{extfaces}
In order to satisfy properties (\ref{C})--(\ref{F}) above, we map the exterior faces of the $A''_i$ to faces lying in the appropriate planes. We also ensure consistency with the definition of $g$ on the end faces given in Subsection~\ref{endfaces}. 

Consider, for example, the exterior face $$L_1 := \partial A''_1 \cap \{ x_1 = 0\}.$$ The boundary of $L_1$ consists of four line segments. Two of these line segments also lie in end faces. The definition of $g$ in Subsection~\ref{endfaces} maps these line segments to line segments. We use an affine map to take the other two edges of $L_1$ to line segments.

All four images of the edges of $L_1$ lie in the plane $\{ x_1 = 0 \}$, and so we let $L_1'$ be the quadrilateral lying in $\{ x_1 = 0 \}$ bounded by these line segments; $L_1'$ is the quadrilateral $P_1'P_L'T_L'T_1'$ in Figure~\ref{figure:A1}. It is easy to see that both $L_1$ and $L_1'$ are compact sets, and that the interiors of each are star domains with non-tangential star centres. Hence we can extend $g$ to a biLipschitz subjection from  $L_1$ to $L'_1$ by an application of Theorem~\ref{theorem:construction}.

It is straightforward to see that the same technique can be used to define $g$ on the other exterior faces, and that properties (\ref{C})--(\ref{F}) are satisfied by this construction.
\begin{remark}\normalfont
Here it is important to note that the four edges of each exterior face are mapped onto genuine quadrilaterals, and not figure-of-eight shapes. It is for this reason that we previously mapped the upper face of $A'$ to the particular ``folded'' arrangement shown shaded in Figure~\ref{fig:inter}. 
\end{remark} 
\subsubsection{The interior faces}
We define $g$ on the interior faces of the $A''_i$ so that the interior of each $B''_i$ is a star domain with a non-tangential star centre. We also ensure consistency with the definition of $g$ on the end faces given in Subsection~\ref{endfaces}, and on the exterior faces give in Subsection~\ref{extfaces}.

Consider, for example, the interior face $$L_2 := \partial A''_1 \cap \{ x_2 = 1\}.$$ Divide $L_2$ into two triangles by adding a line segment from $X_1$ to $T_L$. The boundaries of these two triangles consist, in total, of five line segments. Two of these line segments also lie in end faces. The definition of $g$ in Subsection~\ref{endfaces} maps these line segments to line segments. The line segment $L_2 \cap \{ x_1 = 0\}$ also lies in an exterior face. The definition of $g$ in Subsection~\ref{extfaces} maps this line segment to a line segment. Finally, we use an affine map to take the other two edges to line segments.

Thus we have described how $g$ maps the boundaries of the two triangles $T_1 X_1 T_L$ and $T_L X_1 X_L$ onto the boundaries of two image triangles. We can now apply Theorem~\ref{theorem:construction} to extend $g$ from the boundary of each to the whole triangle, and therefore to the whole of $L_2$. 

It is straightforward to see that the same technique can be used to define $g$ on the other interior faces.

\subsubsection{Defining the function $g$ inside the $A''_i$}
\begin{figure}[ht]
	\centering
	\includegraphics[width=14cm,height=10cm]{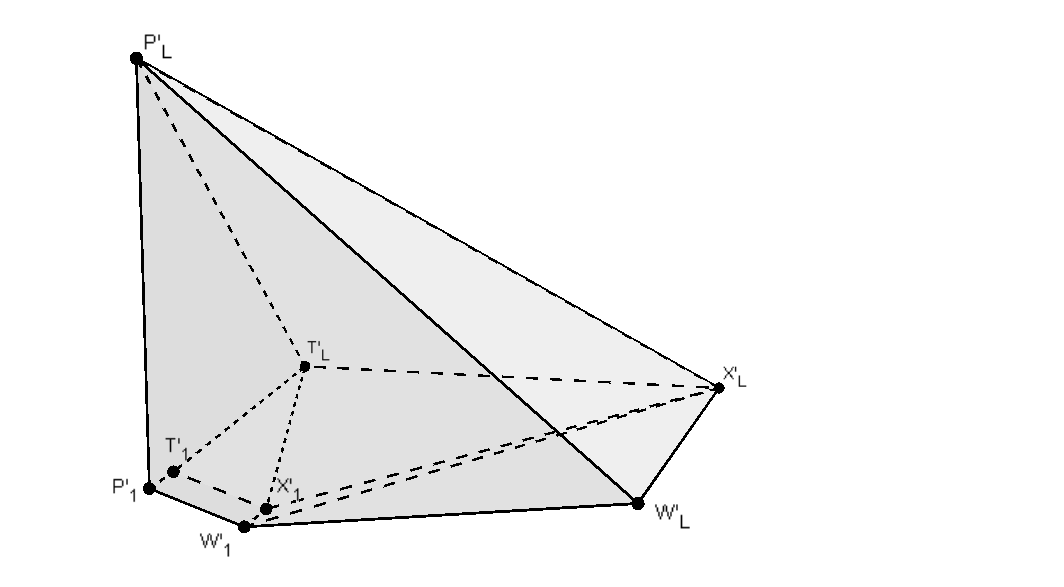}
	\caption{The image of $A''_1$. As before, $P'_1$ denotes $g(P_1)$, for example.}
	\label{figure:A1}
\end{figure}
Suppose that $i \in \{1,2,3,4\}$. Let $B''_i$ be the compact set bounded by the images of the boundary faces of $A''_i$. It can be seen that the interior of $B''_i$ is a star domain with a non-tangential star centre; see Figure \ref{figure:A1}, and we note that the images of $A''_2$, $A''_3$ and $A''_4$ are similar to this. The non-tangential star centres can be taken to be close to the vertices with large third coordinate; in other words near the images of $P_\firstbound$, $S_\firstbound$, $Q_\firstbound$ and $R_\firstbound$. It follows that, by an application of Theorem~\ref{theorem:construction}, we can extend $g$ to a biLipschitz surjection from $A''_i$ to $B''_i$. 

Since $g$ is biLipschitz, it follows -- as earlier -- that $g$ is quasiconformal and is either orientation-preserving or orientation-reversing. It is easy to see that $g$ is, in fact, orientation-preserving in each case and is therefore quasiregular.

This completes the definition of $g$ in $A''$, and so completes the proof of Theorem~\ref{Texistshalf}.
%
%
%
%
%
%
\section{Proof of Theorem~\ref{Texists}}
\label{Sexists}
In this section we use the function $g$ constructed in Section~\ref{Sexistshalf} to prove Theorem~\ref{Texists}. The function $f$ is defined as follows. By construction, the third component of $g(x)$ is bounded above when $x \in \{ x_3 \leq \firstbound \}$, where $\firstbound$ is the constant from Proposition~\ref{prop:Fbeam}. Hence we can choose $\secondbound>0$ sufficiently large that the quasiregular map
\begin{equation}
\label{mainfdef}
f(x) := g(x) - (0, 0, \secondbound),
\end{equation}
satisfies
\begin{equation}
\label{fprop}
f(\{ x_3 \leq \firstbound\}) \subset \{ x_3 < 0 \}.
\end{equation}

Recall the definition (\ref{Bnmeq}) of a fundamental half-beam. We require the following. Let $\mathcal{L}$ be the union of all lines  of the form
\begin{equation}
\label{Ldef}
\{ x_1 = 4n + c, x_2 = 4m + (2-c) \}, \qfor n, m \in\mathbb{Z}, \ c \in \{0,2\}.
\end{equation}
\begin{proposition}
\label{prop:discs}
Suppose that $\xi \in \{ x_3 > \firstbound \}$ is such that $f(\xi) \in \{ x_3 > \firstbound \}$. Suppose that $\delta > 0$, that $B(\xi, \delta)$ is contained in a fundamental half-beam, and that $f(B(\xi, \delta)) \cap \mathcal{L} = \emptyset$. Then, there exists $\xi' \in B(\xi, \delta)$ such that $B(f(\xi'), 2\delta)$ is contained in the intersection of $f(B(\xi, \delta))$ with a fundamental half-beam.
\end{proposition}
\begin{proof}
First we note that $B(f(\xi), 32\delta) \subset f(B(\xi, \delta))$. This claim follows from the fact that, in $\{ x_3 > L\}$, $f$ and $F$ differ only by a translation, where $F$ is the function defined in (\ref{Feq}), and from the second part of Proposition~\ref{prop:Fbeam}. The details are essentially the same as the proof of \cite[Lemma 3.3]{MR3173490}, and are omitted.

The fact that $f(\xi) \in \{ x_3 > \firstbound \}$ implies that there is a hemisphere of $B(f(\xi), 32\delta)$ that lies in $\{x_3 > L\}$. The result then follows from the fact that this hemisphere does not meet $\mathcal{L}$. 
\end{proof}
We also require the following result, which is a generalisation to discrete open maps of \cite[Corollary 2]{MR1642181}. Here if $\phi : \R^d \to \R^d$ is a continuous map, then we say that $\xi \in \R^d$ is a \emph{finite asymptotic value} of $\phi$ if there is a curve $\Gamma : (0, 1) \to \R^d$ such that $\lim_{t \rightarrow 1} \Gamma(t) = \infty$ and  $\lim_{t \rightarrow 1} \phi(\Gamma(t)) = \xi$. 
\begin{lemma}
\label{lemm:asympt}
Suppose that $\phi : \R^d \to \R^d$ is continuous, discrete and open, and has no finite asymptotic value. Let $D \subset \R^d$ be a domain, and suppose that $E$ is a component of $\phi^{-1}(D)$. Then $\phi(E) = D$.
\end{lemma}
\begin{proof}
We prove the contrapositive; if $D\setminus \phi(E) \ne \emptyset$, then $\phi$ has a finite asymptotic value. Accordingly, suppose that $y \in D\setminus \phi(E)$. Let $B_\phi$ be the \emph{branch set} of $\phi$; in other words
$$
B_\phi = \{ x \in \R^d : \phi \text{ is not a local homeomorphism at } x \}.
$$ Choose $x_0 \in E\setminus B_\phi$. Let $\gamma : [0,1] \to D$ be a non-self-intersecting curve, compactly contained in $D$, such that
\begin{itemize}
\item $\gamma(0) = \phi(x_0)$ and $\gamma(1) = y$;
\item $\gamma([0,1)) \cap \phi(B_\phi) = \emptyset$.
\end{itemize}
This is possible because \cite[Lemma 9.8]{MR950174} the topological dimension of $\phi(B_\phi)$ is at most $d-2$. (In fact, we can take $\gamma$ to be a polygonal curve with a finite number of pieces.)

Since $x_0 \notin B_\phi$, the function $\phi$ is injective in a neighbourhood of $x_0$. It follows that, for all sufficiently small $t>0$, there exists a curve $\Gamma_t:[0,t]\to\R^d$ such that $\Gamma_t(0) = x_0$ and $\phi(\Gamma_t(s)) = \gamma(s)$, for $0\leq s \leq t$. Let $t'$ be the supremum of these values of $t$. 

Suppose that $t \in (0, t')$. Since $\Gamma_t \cap B_\phi = \emptyset$, the function $\phi$ is locally injective on $\Gamma_t$. Hence, if $s \in (0, t]$, then $\Gamma_t|_{[0,s]} = \Gamma_s$. It follows that we can define a curve $\Gamma : [0,t') \to \R^d$ by $\Gamma(t) = \Gamma_t(t)$. 

Let $\mathcal{T}$ denote the set of strictly increasing sequences of real numbers in $[0, t')$ with the following property. If $(\tau_n)_{n\isnatural} \in \mathcal{T}$, then $\tau_n \rightarrow t'$ as $n\rightarrow\infty$, and $\Gamma(\tau_n)\rightarrow w$ as $n\rightarrow\infty$, where $w \in \overline{E} \cup \{ \infty\}$. 


Suppose that $(\tau_n)_{n\isnatural}$ and  $(\tau'_n)_{n\isnatural}$ are elements of $\mathcal{T}$. Suppose in addition that $\Gamma(\tau_n)\rightarrow w_1$ and $\Gamma(\tau'_n)\rightarrow w_2$ as $n\rightarrow\infty$, and that $w_1 \ne w_2$. We can assume that $w_1$ is finite. Since $\phi$ is discrete, this implies that there exist $(\tau''_n)_{n\isnatural} \in \mathcal{T}$ and $w_3 \in \overline{E}$ such that $\Gamma(\tau''_n)\rightarrow w_3$ as $n\rightarrow\infty$ and $\phi(w_3) \ne \phi(w_1)$. However, by continuity, $\phi(w_1) = \gamma(t') = \phi(w_3)$. This is a contradiction. 


It follows that either $\lim_{t\rightarrow t'} \Gamma(t) =\infty$, or $\lim_{t\rightarrow t'} \Gamma(t)$ exists and is finite. In the first case $\phi$ has a finite asymptotic value equal to $\gamma(t')$. It remains to show that the second case leads to a contradiction.

Suppose that $w = \lim_{t\rightarrow t'} \Gamma(t)$ is finite. If $w \in \partial E$, then $\gamma$ accumulates on $\partial D$, which is a contradiction. Hence $w \in E$, and so $\gamma(t') = \phi(w) \ne y$. It follows that $0 < t' < 1$, and that $\phi$ is a homeomorphism from a sufficiently small neighbourhood, $U$, of $w$, onto a neighbourhood, $V$, of $\phi(w)$. Hence, there exists $\epsilon >0$ such that $\gamma([t'-\epsilon,t'+\epsilon]) \subset V$. Thus the concatenation of $\Gamma_{t'-\epsilon}$ with $\phi^{-1} \circ \gamma|_{[t'-\epsilon,t'+\epsilon]}$ gives a curve $\Gamma_{t'+\epsilon}$, in contradiction to our choice of $t'$.
%
\end{proof}
\begin{proof}[Proof of Theorem~\ref{Texists}]
Define $H_0 := \{ x_3 < 0 \}$, and let
$$
 H_n := f^{-n}(H_0), \qfor n\isnatural.
$$

Then $\{H_n\}_{n\geq 0}$ is an increasing sequence of sets. Set $U := \bigcup_{n \geq 0} H_n$. \\

First we claim that $U \subset QF(f)$ and that $f^k\rightarrow\infty$ as $k\rightarrow\infty$ locally uniformly in $U$. Suppose that $x \in U$. Then there exists $k\isnatural$ such that $f^k(x) \in H_0$, and so there is a neighbourhood $\Delta$ of $x$ sufficiently small that $f^k(\Delta) \subset H_0$. The claim follows from this fact, recalling the definition of the quasi-Fatou set $QF(f)$ and the facts that $f(H_0) \subset H_0$ and $f$ acts as a translation on $H_0$. 
The fact that $f$ is locally uniformly quasiregular in $U$ follows in the same way. \\

We show next that $U$ is connected. We claim that $H_n$ is connected, for $n\isnatural$. Since $H_0 \subset H_n$, for $n\isnatural$, the fact that $U$ is connected follows from this claim.

We prove the claim by induction. First we show that $H_1 = H_0 \cup H_1$ is connected. Recall the definition of the collection of lines $\mathcal{L}$ from (\ref{Ldef}). These lines lie in $H_1$. Moreover, by a calculation, there exists $\epsilon > 0$ such that 
$$
\mathcal{L}' := \{ x : \operatorname{dist }(x, \mathcal{L}) < \epsilon \} \subset H_1.
$$
Since each component of $\mathcal{L}'$ meets $H_0$, it follows that $H_1$ has a component $V$ which contains $\mathcal{L}'$.

Let $V' \subset H_1$ be a component of $f^{-1}(H_0)$. It can be shown that if $V'$ contains a point $v = (v_1, v_2, v_3)$ that is not in any fundamental half-beam, then $$\{ x_1 = v_1, x_2 = v_2, x_3 \leq v_3 \} \subset V'.$$ In this case $V' = V$. It follows that we can assume that $V'$ is contained in some fundamental half-beam $\mathcal{B}$.

It can be seen that $f$ has no finite asymptotic values. Hence, by Lemma~\ref{lemm:asympt}, $V'$ contains a set, $X$ say, such that $$f(X) = \{ x_1 = 0, x_2 = 0, x_3 < 0\}.$$ It can be shown by a calculation that $X$ meets  $\mathcal{L}'$; this can be seen, for example, by considering, in the fundamental half-beam $\mathcal{B}$, the solution to the equation $$f(x) = (0, 0, c),$$ where $c$ is taken to be large and negative. It follows that $V' = V$, and so $H_1$ is connected, as required.

Now, suppose that $k\isnatural$ and that $H_k$ is connected. It remains to show that $H_{k+1}$ is connected. Suppose that $W$ is a component of $H_{k+1}$. It follows by Lemma~\ref{lemm:asympt} that there is a point $x \in W$ such that $f(x) \in H_{k-1}$, in which case $x \in H_k$. Thus $W$ meets $H_k$, and so $H_{k+1}$ is connected, as required. \\

%
%


Next we show that $U = QF(f)$. Suppose, by way of contradiction, that there exists $\xi_0 \in QF(f)\setminus U$. It follows that there exists $\delta > 0$ sufficiently small that $f^k(B(\xi_0, \delta)) \cap U = \emptyset$, for $k \geq 0$. We deduce that 
$$
f^k(B(\xi_0, \delta)) \subset \{ x_3 > L \} \setminus \mathcal{L}, \qfor k\geq 0.
$$

We can assume, by adjusting $\xi_0$ and $\delta$ if necessary, that $B(\xi_0, \delta)$ is contained in a fundamental half-beam.

Since both $\xi_0$ and $f(\xi_0)$ lie in $\{ x_3 > \firstbound \}$, it follows by Proposition~\ref{prop:discs} that there exists $\xi_1 \in QF(f)\setminus U$ such that $f^k(B(\xi_1, 2\delta)) \cap U = \emptyset$, for $k \geq 0$, and $B(\xi_1, 2\delta)$ is contained in the intersection of $f(B(\xi_0, \delta))$ with a fundamental half-beam.

By repeated application of Proposition~\ref{prop:discs}, we obtain a sequence of points $(\xi_k)_{k\geq 0}$ such that $B(\xi_k, 2^k\delta) \cap U = \emptyset$ and $B(\xi_k, 2^k\delta)$ is contained in a fundamental half-beam, for $k\geq 0$. 
This is a contradiction, since these balls cannot be contained in a fundamental half-beam for $2^k\delta > 1$. 

Finally we show that $U$ is full. Suppose, by way of contradiction, that $U$ is hollow. Since $U$ is periodic, it follows by Theorem~\ref{T2} that there exists $x \in U \cap A(f)$. Hence, by Lemma~\ref{lemm:inA}, equation (\ref{xinAeq}) holds. This is a contradiction, since $f^k(x) \in H_0$, for all sufficiently large $k$, and $f$ acts as a translation in $H_0$.
\end{proof} 
\begin{remarks}\normalfont
\begin{enumerate}
\item If $L'' > 0$ is sufficiently large, then $J(f)$ contains all lines of the form $$\{ x_1 = 4n + c, x_2 = 4m + c, x_3 \geq L'' \}, \qfor n, m \in\mathbb{Z}, \ c \in \{0, 2\}.$$ It seems possible that the techniques of \cite{MR2730579} could be used to prove that $J(f)$ is, in fact, a ``Cantor bouquet'' consisting of an uncountable number of curves, but we have not tried to do this. \\

\item Evdoridou \cite{Vasso1} studied the function $h$ defined in (\ref{Fatoueq}), and showed that $I(h)$, but not $A(h)$, has a structure known as a {\spw}; for a definition of a {\spw} see \cite{MR3215194}. This gives a positive answer to a question of Rippon and Stallard \cite[Question 3]{Rippon01102012}, which was previously answered using a more complicated construction in \cite{MR3077875}. In three dimensions, it is straightforward to see that $A(f)$ is not a {\spw}, where $f$ is the the quasiregular map defined in (\ref{mainfdef}). It seems natural to ask if it is the case that $I(f)$ is a {\spw}.
\end{enumerate}
\end{remarks}
%
%
%
%
%
%
\section{Additional examples}
\label{Sexample}
Our first example shows that, in $\R^2$, the inequality (\ref{Tspeedeq1}) is best possible.
\begin{example}\normalfont
\label{Example1}
Let $\phi : \R^2 \to \R^2$ be the orientation-preserving quasiconformal map $\phi(x) := |x|x$. With a very slight abuse of notation, let $f : \R^2 \to \R^2$ be the quasiregular map defined by $f := h \circ \phi,$ where $h$ is the map defined in (\ref{Fatoueq}).

It is straightforward to show first that $K_I(\phi) = 2$, and that it follows from this that $K_I(f) = 2$. It can be seen that there is a periodic component $U$ of $QF(f)$, containing a right half-plane $H$, such that 
$$
\log \log |f^k(x)| \sim k \log 2 = k \log K_I(f) \text{ as } k\rightarrow\infty, \qfor x \in H.
$$  

We show that $U$ is full as follows. Clearly $U$ meets $A(f)^c$. If $U$ is hollow, then it follows from Theorem~\ref{T2} that $U$ also meets $A(f)$. This is a contradiction since,  by the final part of Lemma~\ref{lemm:JA}, we have that $J(f) = \partial A(f)$. Hence $U$ is full.

\end{example}
Our second example shows that, in $\R^3$, the inequality (\ref{Tspeedeq1}) is best possible up to a constant.
\begin{example}\normalfont
\label{Example2}
Let $\phi : \R^3 \to \R^3$ be the orientation-preserving quasiconformal map $\phi(x) := |x|x$, let $f$ be the quasiregular map defined in (\ref{mainfdef}), and let $G : \R^3 \to \R^3$ be the quasiregular map $G := f \circ \phi.$ 

It can be seen that $G$ has a periodic domain $U$, containing a half-space $H$, in which 
\begin{equation*}
\log\log|G^{k}(x)| \sim k \log 2 \text{ as } k\rightarrow\infty, \qfor x \in H.
\end{equation*}

The fact that $U$ is full follows in exactly the same way as in Example~\ref{Example1}.
\end{example}
Our final example shows that there is a quasiregular map of $\R^3$ with a periodic domain that meets, but is not contained in, the escaping set.
\begin{example}\normalfont
\label{Example3}
Let $f$ be the quasiregular map defined in (\ref{mainfdef}). Let $C$ be an open ball in $\{ x_3 < 0\}$ sufficiently large that $C \cap f(C) \ne \emptyset$. Choose a point $x_0 \in C \cap f(C)$. Let $\phi : \R^3 \to \R^3$ be an orientation-preserving quasiconformal map such that $\phi(x) = x,$ for $x \notin f(C)$, and $\phi(f(x_0)) = x_0$. It is straightforward to see that such a $\phi$ exists; for example, by an application of Theorem~\ref{theorem:construction}.

Let $g = \phi \circ f$. It can be seen that $QF(g) = QF(f)$, and so $QF(g)$ consists of a full domain $U$. It can also be seen that $U \cap I(g) \ne \emptyset$, and that $x_0 \in U \setminus I(g)$.
\end{example}
%
%
%
%
%
%

\bibliographystyle{acm}
\bibliography{../../../Research.References}
\end{document}